\documentclass[letter,reqno,10pt,twoside]{amsart}
\usepackage{amsmath,amsthm,amsfonts,amssymb,euscript,mathrsfs,graphics,color,latexsym,marginnote}
\usepackage[dvips]{graphicx}
\usepackage[margin=1in]{geometry}
\usepackage{hyperref}
\usepackage[toc,page]{appendix}
\usepackage{relsize}
\usepackage[shortlabels]{enumitem}
\usepackage[toc,page]{appendix}
\usepackage{float}

\usepackage[dvipsnames]{xcolor}

\usepackage{hyperref}
\hypersetup{colorlinks=true, pdfstartview=FitV,linkcolor=blue!70!black,citecolor=red!70!black, urlcolor=green!60!black}
\definecolor{labelkey}{rgb}{0.6,0,0} 

\usepackage{pgfplots}

\newtheorem{theorem}{Theorem}[section]
\newtheorem{lemma}[theorem]{Lemma}
\newtheorem{proposition}[theorem]{Proposition}

\newtheorem{remark}[theorem]{Remark}

\makeatletter
\renewcommand \theequation {%
\ifnum \c@section>\z@ \@arabic\c@section.%
\fi\@arabic\c@equation} \@addtoreset{equation}{section}
\makeatother

\DeclareMathOperator*{\esssup}{ess\,sup}

\providecommand{\abs}[1]{\left\vert#1\right\vert}
\providecommand{\babs}[1]{\big\vert#1\big\vert}
\providecommand{\nm}[1]{\left\Vert#1\right\Vert}
\providecommand{\nnm}[1]{\left\vert\kern-0.25ex\left\vert\kern-0.25ex\left\vert#1\right\vert\kern-0.25ex\right\vert\kern-0.25ex\right\vert}

\providecommand{\br}[1]{\left\langle #1 \right\rangle}
\providecommand{\bbr}[1]{\Big\langle #1 \Big\rangle}
\providecommand{\brx}[1]{\left\langle #1 \right\rangle_x}

\providecommand{\brw}[1]{\left\langle #1 \right\rangle_w}

\providecommand{\tnm}[1]{\left\Vert#1\right\Vert_{L^2}}
\providecommand{\lnm}[1]{\left\Vert#1\right\Vert_{L^{\infty}}}
\providecommand{\tnms}[2]{\left\vert#1\right\vert_{L^2_{#2}}}
\providecommand{\lnms}[2]{\left\vert#1\right\vert_{L^{\infty}_{#2}}}

\def\ud{\mathrm{d}}

\def\p{\partial}
\def\ls{\lesssim}
\def\gs{\gtrsim}

\def\rt{\rightarrow}
\def\r{\mathbb{R}}
\def\no{\nonumber}
\def\ue{\mathrm{e}}

\def\ds{\displaystyle}

\def\S{\mathbb{S}}

\def\nx{\nabla_x}

\def\s{\S}
\def\e{\varepsilon}
\def\d{\delta}
\def\vw{w}
\def\vx{x}
\def\vn{n}

\def\u{U}
\def\bu{\overline{\u}}
\def\uu{U^B}

\def\re{R}
\def\bre{\overline{\re}}
\def\ire{\re-\bre}
\def\ss{S}
\def\g{h}

\def\vr{r}
\def\vt{\varsigma}
\def\kk{\kappa}

\def\bl{\Phi}
\def\bbl{\overline{\bl}}
\def\bll{\Psi}
\def\bbll{\overline{\bll}}
\def\ch{\widetilde{\chi}}

\def\oo{o(1)}

\def\test{\xi}
\def\vr{\mathbf{r}}
\def\mn{\mu}
\def\vww{\mathfrak{w}}
\def\pl{L}
\def\ua{u_a}
\def\weak{\mathfrak{g}}

\def\N{r}
\def\NN{s}

\begin{document}

\title{$L^2$ Diffusive Expansion For Neutron Transport Equation}

\author[Y. Guo]{Yan Guo}
\address[Y. Guo]{
   \newline\indent Division of Applied Mathematics, Brown University}
\email{yan\_guo@brown.edu}
\thanks{Y. Guo was supported by NSF Grant DMS-2106650.}

\author[L. Wu]{Lei Wu}
\address[L. Wu]{
   \newline\indent Department of Mathematics, Lehigh University}
\email{lew218@lehigh.edu}
\thanks{L. Wu was supported by NSF Grant DMS-2104775.}

\date{}

\subjclass[2020]{Primary 35Q49, 82D75; Secondary 35Q62, 35Q20}

\keywords{non-convex domains, transport equation, diffusive limit}

\maketitle

\begin{abstract}
Grazing set singularity leads to a surprising
counter-example and breakdown \cite{AA003} of the classical mathematical theory for $L^{\infty}$ diffusive expansion \eqref{final 15} of neutron transport equation with in-flow boundary condition in term of the Knudsen number $\e$, one of the
most classical problems in the kinetic theory. Even though a satisfactory new theory has been established by constructing new boundary layers with favorable $\e$-geometric correction for convex domains \cite{AA003, AA007, AA009, AA014, AA016}, the severe grazing singularity from non-convex domains has prevented any positive mathematical progress. We develop a novel and optimal $L^2$ expansion theory for general domain (including non-convex domain) by discovering a surprising $\e^{\frac{1}{2}}$ gain for the average of remainder.
\end{abstract}

\pagestyle{myheadings} \thispagestyle{plain} \markboth{Y. GUO, L. WU}{DIFFUSIVE EXPANSION OF NEUTRON TRANSPORT EQUATION}

\setcounter{tocdepth}{1}
\tableofcontents

\section{Introduction}

\subsection{Problem Formulation}

We consider the steady neutron transport equation in a three-dimensional $C^3$ bounded domain (convex or non-convex) with in-flow boundary condition. In the spatial domain $\Omega\ni\vx=(x_1,x_2,x_3)$ and the velocity domain
$\s^2\ni\vw=(w_1,w_2,w_3)$, the neutron density $u^{\e}(\vx,\vw)$
satisfies
\begin{align}\label{transport}
\left\{
\begin{array}{l}\displaystyle
\vw\cdot\nabla_x u^{\e}+\e^{-1}\Big(u^{\e}-\overline{u^{\e}}\Big)=0\ \ \text{in}\ \ \Omega\times\s^2,\\\rule{0ex}{1.2em}
u^{\e}(\vx_0,\vw)=g(\vx_0,\vw)\ \ \text{for}\
\ \vw\cdot\vn<0\ \ \text{and}\ \ \vx_0\in\p\Omega,
\end{array}
\right.
\end{align}
where $g$ is a given function denoting the in-flow data,
\begin{align}\label{average}
\overline{u^{\e}}(\vx):=\frac{1}{4\pi}\int_{\s^2}u^{\e}(\vx,\vw)\ud{\vw},
\end{align}
$\vn$ is the outward unit normal vector, with the Knudsen number $0<\e\ll1$. We intend to study the asymptotic behavior of $u^{\e}$ as $\e\rt0$.

Based on the flow direction, we can divide the boundary $\gamma:=\big\{(\vx_0,\vw):\ \vx_0\in\p\Omega,\vw\in\s^2\big\}$ into the incoming boundary $\gamma_-$, the outgoing boundary $\gamma_+$, and the grazing set $\gamma_0$ based on the sign of $\vw\cdot\vn(\vx_0)$. In particular, the boundary condition of \eqref{transport} is only given on $\gamma_{-}$.

\subsection{Normal Chart near Boundary}\label{sec:geometric-setup}

We follow the approach in \cite{AA009, AA016} to define the geometric quantities, and the details can be found in Section \ref{sec:boundary-layer}. For smooth manifold $\p\Omega$, there exists an orthogonal curvilinear coordinates system $(\iota_1,\iota_2)$ such that the coordinate lines coincide with the principal directions at any $\vx_0\in\p\Omega$.
Assume $\p\Omega$ is parameterized by $\vr=\vr(\iota_1,\iota_2)$. Let the vector length be $\pl_i:=\abs{\p_{\iota_i}\vr}$ and unit vector $\vt_i:=\pl_i^{-1}\p_{\iota_i}\vr$ for $i=1,2$.

Consider the corresponding new coordinate system $(\mn, \iota_1,\iota_2)$, where $\mn$ denotes the normal distance to the boundary surface $\p\Omega$, i.e.
\begin{align}
\vx=\vr-\mn\vn.
\end{align}
Define the orthogonal velocity substitution for $\vww:=(\varphi,\psi)$ as
\begin{align}\label{velocity}
-\vw\cdot\vn=\sin\varphi,\quad
\vw\cdot\vt_1=\cos\varphi\sin\psi,\quad
\vw\cdot\vt_2=\cos\varphi\cos\psi.
\end{align}
Finally, we define the scaled normal variable $\eta=\dfrac{\mn}{\e}$, which implies $\dfrac{\p}{\p\mn}=\dfrac{1}{\e}\dfrac{\p}{\p\eta}$.

\subsection{Asymptotic Expansion and Remainder Equation}

We seek a solution to \eqref{transport} in the form 
\begin{align}\label{expand}
    u^{\e}=&\u+\uu+\re=\left(\u_0+\e\u_1+\e^2\u_2\right)+\uu_0+\re,
\end{align}
where the interior solution is
\begin{align}\label{expand 1}
\u(x,w):= \u_0(x,w)+\e\u_1(x,w)+\e^2\u_2(x,w),
\end{align}
and the boundary layer is
\begin{align}\label{expand 2}
\uu(\eta,\iota_1,\iota_2,\vww):= \uu_0(\eta,\iota_1,\iota_2,\vww).
\end{align}
Here $\u_0$, $\u_1$, $\u_2$ and $\uu_0$ are constructed in Section \ref{sec:interior} and Section \ref{sec:boundary-layer}, and $\re(x,v)$ is the remainder.

\subsection{Literature}

The study of the neutron transport equation in bounded domains, has attracted a lot of attention since the dawn of the atomic age. Besides its significance in nuclear sciences and medical imaging, neutron transport equation is usually regarded as a linear prototype of the more important yet more complicated nonlinear Boltzmann equation, and thus, is an ideal starting point to develop new theories and techniques. We refer to
\cite{Larsen1974=, Larsen1974, Larsen1975, Larsen1977, Larsen.D'Arruda1976, Larsen.Habetler1973, Larsen.Keller1974, Larsen.Zweifel1974, Larsen.Zweifel1976} for the formal expansion with respect to $\e$ and explicit solution. The discussion on bounded domain and half-space cases can be found in \cite{Bensoussan.Lions.Papanicolaou1979, Bardos.Santos.Sentis1984, Bardos.Phung2017, Bardos.Golse.Perthame1987, Bardos.Golse.Perthame.Sentis1988, Li.Lu.Sun2015, Li.Lu.Sun2015==, Li.Lu.Sun2017}.

The classical boundary layer of neutron transport equation dictates that $\uu_0(\eta,\iota_1,\iota_2,\vww)$ satisfies the Milne problem
\begin{align}\label{final 12}
    \sin\varphi\frac{\p \uu_0}{\p\eta}+\uu_0-\overline{\uu_0}=0.
\end{align}
From the formal expansion in $\e$ (see \eqref{expand 6}),  it is natural to expect
the remainder estimate \cite{Bensoussan.Lions.Papanicolaou1979}
\begin{align}\label{final 15}
    \lnm{\re}\ls\e.
\end{align}
Even though this is valid for domains with flat boundary, a counter-example is constructed \cite{AA003} so that \eqref{final 15} is invalid for a 2D disk. This is due to the grazing set singularity.

To be more specific, in order to show the remainder estimates \eqref{final 15}, the higher-order boundary layer expansion $\uu_1\in L^{\infty}$ is necessary, which further requires $\p_{\iota_i}\uu_0\in L^{\infty}$. Nevertheless, though $\uu_0\in L^{\infty}$, it is shown that the normal derivative $\p_{\eta}\uu_0$ is singular at the grazing set $\varphi=0$. Furthermore, this singularity $\p_{\eta}\uu_0\notin L^{\infty}$ will be transferred to $\p_{\iota_i}\uu_0\notin L^{\infty}$. A careful construction of boundary data \cite{AA003} justifies this invalidity, i.e. both the method and result of the boundary layer \eqref{final 12} are problematic.

A new construction of boundary layer \cite{AA003} based on the $\e$-Milne problem with geometric correction for $\widetilde{\uu_0}(\eta,\iota_1,\iota_2,\vww)$
\begin{align}\label{final 13}
    \sin\varphi\frac{\p \widetilde{\uu_0}}{\p\eta}-\frac{\e}{1-\e\eta}\cos\varphi\frac{\p \widetilde{\uu_0}}{\p\varphi}+\widetilde{\uu_0}-\overline{\widetilde{\uu_0}}=0
\end{align}
has been shown to provide the satisfactory characterization of the $L^{\infty}$ diffusive expansion in 2D disk domains. With more detailed regularity analysis and boundary layer decomposition techniques for \eqref{final 13}, such result has been generalized to 2D/3D smooth convex domains \cite{AA007,AA009,AA014,AA016} and even 2D annulus domain \cite{AA006}.   

In non-convex domains, the boundary layer with geometric correction is essentially 
\begin{align}\label{final 14}
    \sin\varphi\frac{\p \widetilde{\uu_0}}{\p\eta}-\frac{\e}{1+\e\eta}\cos\varphi\frac{\p \widetilde{\uu_0}}{\p\varphi}+\widetilde{\uu_0}-\overline{\widetilde{\uu_0}}=0
\end{align}
Compared to \eqref{final 13}, this sign flipping dramatically changes the characteristics.
\begin{figure}[H]
\begin{minipage}[t]{0.5\linewidth}
\centering
\includegraphics[width=2.5in]{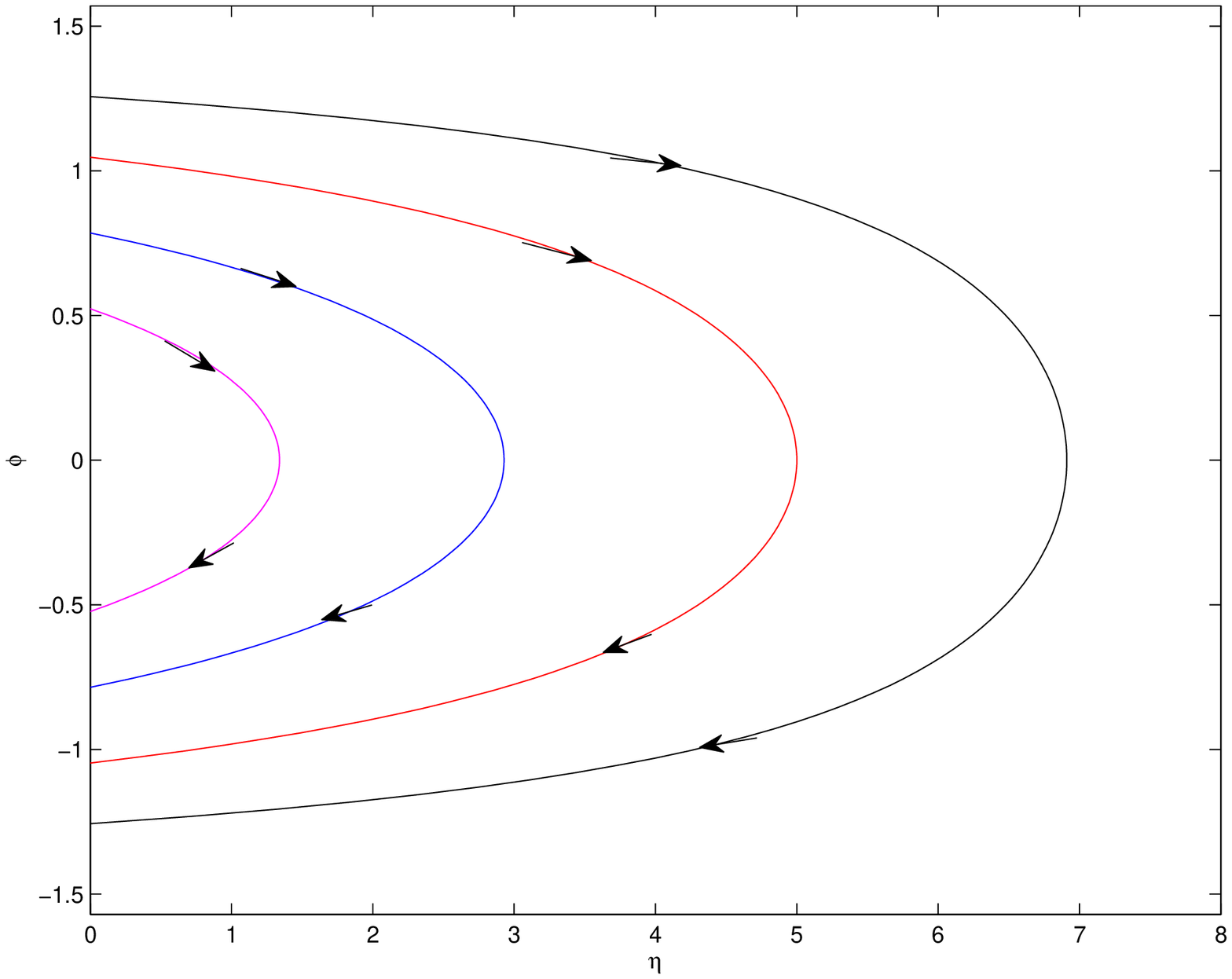}
\caption{Characteristics in Convex Domains} \label{fig 1}
\end{minipage}%
\begin{minipage}[t]{0.5\linewidth}
\centering
\includegraphics[width=2.5in]{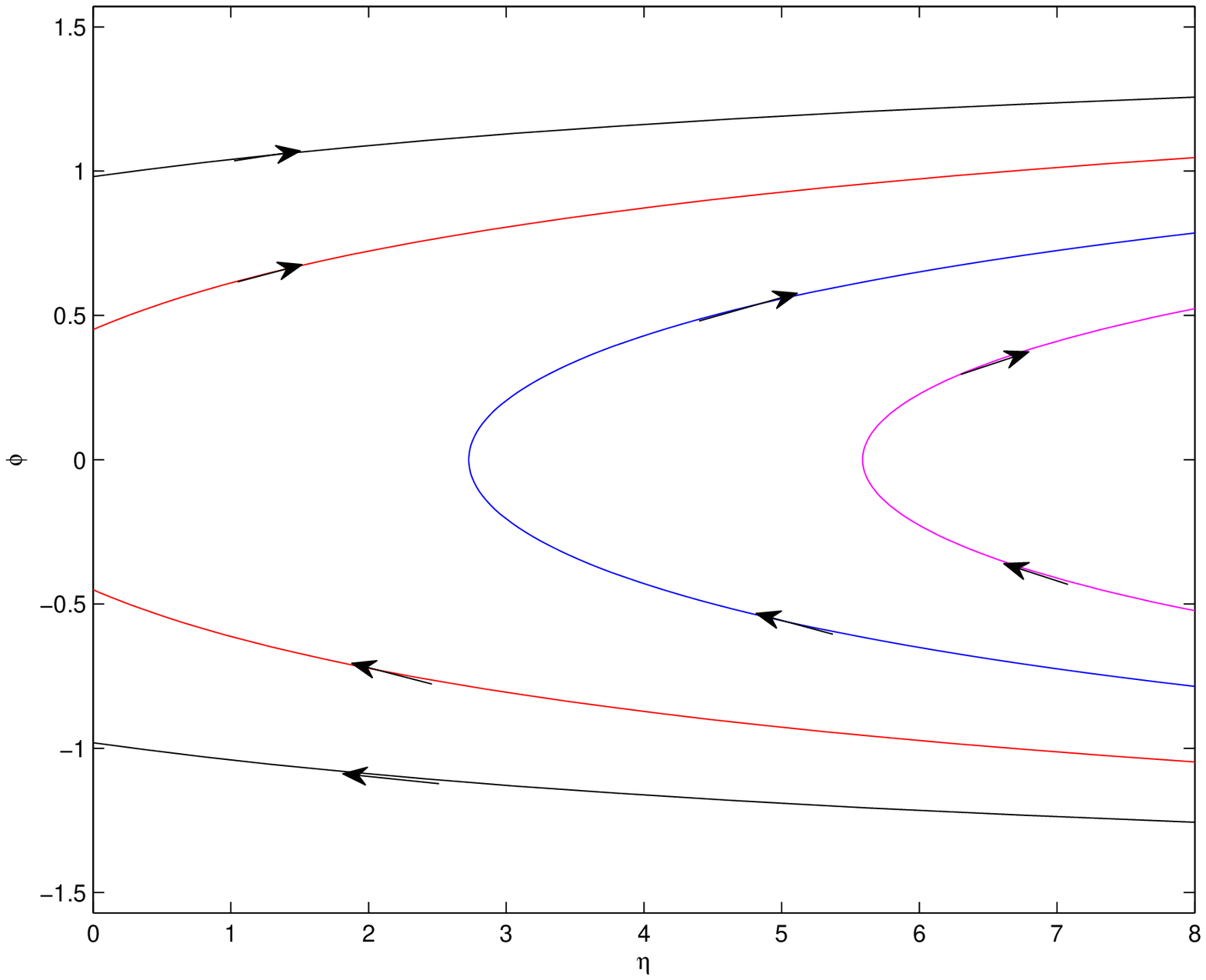}
\caption{Characteristics in Non-Convex Domains} \label{fig 2}
\end{minipage}
\end{figure}
In Figure \ref{fig 1} and Figure \ref{fig 2} \cite{AA006}, the horizontal direction represents the scaled normal variable $\eta$ and the vertical direction represents the velocity $\varphi$. There exists a ``hollow" region in Figure \ref{fig 2} that the characteristics may never track back to the left boundary $\eta=0$ and $\varphi>0$, making the $W^{1,\infty}$ estimates impossible and thus preventing higher-order boundary layer expansion. 

In this paper, we will employ a fresh approach to design a cutoff boundary layer without the geometric correction and justify the $L^2$ diffusive expansion in smooth non-convex domains.

\subsection{Notation and Convention}

Let $\brw{\ \cdot\ ,\ \cdot\ }$ denote the inner product for $w\in\s^2$, $\brx{\ \cdot\ ,\ \cdot\ }$ for $x\in\Omega$, and $\br{\ \cdot\ ,\ \cdot\ }$ for $(x,w)\in\Omega\times\s^2$. Also, let $\br{\ \cdot\ ,\ \cdot\ }_{\gamma_{\pm}}$ denote the inner product on $\gamma_{\pm}$ with measure $\ud\gamma:=\abs{w\cdot n}\ud w\ud S_x=\abs{\sin\varphi}\cos\varphi\ud\vww\ud S_x$. 
Denote the bulk and boundary norms
\begin{align}
    \tnm{f}:=\left(\iint_{\Omega\times\s^2}\abs{f(x,w)}^2\ud w\ud x\right)^{\frac{1}{2}},\quad \tnms{f}{\gamma_{\pm}}:=\left(\int_{\gamma_{\pm}}\abs{f(x,w)}^2\ud\gamma\right)^{\frac{1}{2}}.
\end{align}
Define the $L^{\infty}$ norms
\begin{align}
    \lnm{f}:=\esssup_{(x,w)\in\Omega\times\s^2}\babs{f(x,w)},\quad
    \lnms{f}{\gamma_{\pm}}:=\esssup_{(x,w)\in\gamma_{\pm}}\babs{f(x,w)}.
\end{align}
Let $\nm{\cdot}_{W^{k,p}_x}$ denote the usual Sobolev norm for $x\in\Omega$ and  $\abs{\cdot}_{W^{k,p}_x}$ for $x\in\p\Omega$, and $\nm{\cdot}_{W^{k,p}_xL^q_w}$ denote $W^{k,p}$ norm for $x\in\Omega$ and $L^q$ norm for $\vw\in\s^2$. The similar notation also applies when we replace $L^q$ by $L^q_{\gamma}$. When there is no possibility of confusion, we will ignore the $(x,w)$ variables in the norms.

Throughout this paper, $C>0$ denotes a constant that only depends on
the domain $\Omega$, but does not depend on the data or $\e$. It is
referred as universal and can change from one inequality to another. We write $a\ls b$ to denote $a\leq Cb$ and $a\gs b$ to denote $a\geq Cb$. Also, we write $a\simeq b$ if $a\ls b$ and $a\gs b$.
We will use $\oo$ to denote a sufficiently small constant independent of the data.

\subsection{Main Results}

\begin{theorem}\label{main theorem}
Under the assumption 
\begin{align}\label{assumption}
    \abs{g}_{W^{3,\infty}L^{\infty}_{\gamma_-}}\ls 1,
\end{align}
there exists a unique solution $u^{\e}(\vx,\vw)\in L^{\infty}(\Omega\times\s^2)$ to \eqref{transport}. Moreover, the solution obeys the estimate
\begin{align}\label{main}
\tnm{u^{\e}-\u_0}\ls\e^{\frac{1}{2}}.
\end{align}
Here $\u_0(\vx)$ satisfies the Laplace equation with Dirichlet boundary condition
\begin{align}
\left\{
\begin{array}{l}
\Delta_x\u_0(\vx)=0\ \ \text{in}\
\ \Omega,\\\rule{0ex}{1.2em}
\u_0(\vx_0)=\bl_{\infty}(\vx_0)\ \ \text{on}\ \
\p\Omega,
\end{array}
\right.
\end{align}
in which $\bl_{\infty}(\iota_1,\iota_2)=\bl_{\infty}(\vx_0)$ for $\vx_0\in\p\Omega$ is given by solving the Milne problem for  $\bl(\eta,\iota_1,\iota_2,\vww)$
\begin{align}
\left\{
\begin{array}{l}
\sin\varphi\dfrac{\p \bl }{\p\eta}+\bl -\bbl =0,\\\rule{0ex}{1.2em}
\bl (0,\iota_1,\iota_2,\vww)=g(\iota_1,\iota_2,\vww)\ \ \text{for}\ \
\sin\varphi>0,\\\rule{0ex}{1.5em}
\ds\lim_{\eta\rt\infty}\bl (\eta,\iota_1,\iota_2,\vww)=\bl_{\infty}(\iota_1,\iota_2).
\end{array}
\right.
\end{align}
\end{theorem}

\begin{remark}
    In \cite{AA003,AA014,AA016} for 2D/3D convex domains, as well as \cite{AA006} for 2D annulus domain, it is justified that for any $0<\d\ll1$
    \begin{align}
    \tnm{u^{\e}-\widetilde{\u_0}-\widetilde{\uu_0}}\ls\e^{\frac{5}{6}-\delta},
    \end{align}
    where $\widetilde{\uu_0}(\eta,\iota_1,\iota_2,\vww)$ is the boundary layer with geometric correction defined in \eqref{final 13}, and $\widetilde{\u_0}$ is the corresponding interior solution. \cite[Theorem 2.1]{Li.Lu.Sun2017} reveals that the difference between two types of interior solutions
    \begin{align}
        \tnm{\widetilde{\u_0}-\u_0}\ls\e^{\frac{2}{3}}.
    \end{align}
    Due to the rescaling $\eta=\e^{-1}\mn$, for general in-flow boundary data $g$, the boundary layer $\widetilde{\uu_0}\neq0$ satisfies
    \begin{align}
        \tnm{\widetilde{\uu_0}}\simeq\e^{\frac{1}{2}}.
    \end{align}
    Hence, we conclude that
    \begin{align}
        \tnm{u^{\e}-\u_0}\simeq\e^{\frac{1}{2}}.
    \end{align}
    Therefore, this indicates that \eqref{main} in Theorem \ref{main theorem} achieves the optimal $L^2$ bound of the diffusive approximation.
\end{remark}

\subsection{Methodology}

It is well-known that the key of the remainder estimate is to control $\bre$. In a series of work \cite{AA003,AA006,AA007,AA009,AA014,AA016} based on a $L^{2}\rt L^{\infty }$ framework, it is shown that
\begin{align}
\tnm{\bre}\ls \e^{-1}\tnm{\ire}\ls 1
\end{align}
combined from the expected energy (entropy production) bound for $\e^{-1}\tnm{\ire}$. This bound requires the next-order $\e$ expansion of boundary layer approximation, which is impossible for non-convex domains, and barely possible by the new boundary layer theory with the $\e$-geometric correction. The key improvement in our work is
\begin{align}
\tnm{\bre}\ls \e^{\frac{1}{2}}
\end{align}
which is a consequence of the following conservation law for test function $\test(\vx)$ satisfying $-\Delta_x\test=\bre$ and $\test\big|_{\p\Omega}=0$:
\begin{align}\label{kk 01}
    -\bbr{\re,w\cdot\nx\test}=-\bbr{\ire,w\cdot\nx\test}=\bbr{\ss, \test},
\end{align}
where $\bbr{\bre,w\cdot\nx\test}=0$ thanks to the oddness. This conservation law exactly cancels
the worst contribution of $\e^{-1}\tnm{\ire}$ in $\tnm{\bre}$ estimate, which comes from taking test function $w\cdot\nx\test$
\begin{align}\label{kk 02}
    \int_{\gamma}\re\big(w\cdot\nx\test\big)(\vw\cdot n)-\bbr{\re,w\cdot\nx\big(w\cdot\nx\test\big)}+\e^{-1}\bbr{\ire, w\cdot\nx\test}=\bbr{\ss, w\cdot\nx\test}.
\end{align}
Such a key cancellation produces an extra crucial gain of $\e^{\frac{1}{2}}$.  We then conclude the remainder estimate without any further expansion of the (singular) boundary layer approximation.

In addition, we construct a new cut-off boundary layer near $\varphi=0$ to avoid the singularity, and are able to perform delicate and precise estimates to control the resulting complex forcing term $\ss$ (see \eqref{d:s0}--\eqref{d:s3}), in terms of the desired order $\e$ for closure.

\section{Asymptotic Analysis}

\subsection{Interior Solution}\label{sec:interior}

Inserting \eqref{expand 1} into \eqref{transport} and comparing the order of $\e$, following the analysis in \cite{AA009, AA016}, we deduce that
\begin{align}\label{expand 3}
&\u_0=\bu_0,\quad
\Delta_x\bu_0=0,\\\label{expand 4}
&\u_1=\bu_1-\vw\cdot\nx\u_{0},\quad
\Delta_x\bu_1=0,\\\label{expand 4'}
&\u_2=\bu_2-\vw\cdot\nx\u_{1},\quad
\Delta_x\bu_2=0.
\end{align}
We need the boundary layer to determine the boundary conditions for $\bu_0$, $\bu_1$ and $\bu_2$.

\subsection{Boundary Layer}\label{sec:boundary-layer}

\subsubsection{Geometric Substitutions}

The construction of boundary layer requires a local description in a neighborhood of the physical boundary $\p\Omega$. We follow the procedure in \cite{AA009, AA016}:\\

\paragraph{\underline{Substitution 1: Spacial Substitution}}
Following the notation in Section \ref{sec:geometric-setup}, under the coordinate system $(\mu,\iota_1,\iota_2)$, we have
\begin{align}
\vw\cdot\nx=-(\vw\cdot\vn)\frac{\p}{\p\mu}-\frac{\vw\cdot\vt_1}{\pl_1(\kk_1\mu-1)}\frac{\p}{\p\iota_1}-\frac{\vw\cdot\vt_2}{\pl_2(\kk_2\mu-1)}\frac{\p}{\p\iota_2},
\end{align}
where $\kk_i(\iota_1,\iota_2)$ for $i=1,2$ is the principal curvature.\\

\paragraph{\underline{Substitution 2: Velocity Substitution}}
Under the orthogonal velocity substitution \eqref{velocity}
for $\varphi\in\left[-\dfrac{\pi}{2},\dfrac{\pi}{2}\right]$ and $\psi\in[-\pi,\pi]$, we have
\begin{align}\label{expand 5}
\vw\cdot\nx=&\sin\varphi\frac{\p}{\p\mu}-\bigg(\frac{\sin^2\psi}{R_1-\mu}+\frac{\cos^2\psi}{R_2-\mu}\bigg)\cos\varphi\frac{\p}{\p\varphi}+\frac{\cos\varphi\sin\psi}{\pl_1(1-\kk_1\mu)}\frac{\p}{\p\iota_1}+\frac{\cos\varphi\cos\psi}{\pl_2(1-\kk_2\mu)}\frac{\p}{\p\iota_2}\\
&+\frac{\sin\psi}{R_1-\mu}\bigg\{\frac{R_1\cos\varphi}{\pl_1\pl_2}\bigg(\vt_1\cdot\Big(\vt_2\times\big(\p_{\iota_1\iota_2}\vr\times\vt_2\big)\Big)\bigg)
-\sin\varphi\cos\psi\bigg\}\frac{\p}{\p\psi}\no\\
&
-\frac{\cos\psi}{R_2-\mu}\bigg\{\frac{R_2\cos\varphi}{\pl_1\pl_2}\bigg(\vt_2\cdot\Big(\vt_1\times\big(\p_{\iota_1\iota_2}\vr\times\vt_1\big)\Big)\bigg)
-\sin\varphi\sin\psi\bigg\}\frac{\p}{\p\psi},\no
\end{align}
where $R_i=\kk_i^{-1}$ represents the radius of curvature. Note that the Jacobian $\ud\vw=\cos\varphi\ud\varphi\ud\psi$ will be present when we perform integration.\\

\paragraph{\underline{Substitution 3: Scaling Substitution}}
Considering the scaled normal variable $\eta=\e^{-1}\mu$, we have 
\begin{align}\label{expand 6}
\vw\cdot\nx=&\e^{-1}\sin\varphi\frac{\p}{\p\eta}-\bigg(\frac{\sin^2\psi}{R_1-\e\eta}+\frac{\cos^2\psi}{R_2-\e\eta}\bigg)\cos\varphi\frac{\p}{\p\varphi}+\frac{R_1\cos\varphi\sin\psi}{\pl_1(R_1-\e\eta)}\frac{\p}{\p\iota_1}+\frac{R_2\cos\varphi\cos\psi}{\pl_2(R_2-\e\eta)}\frac{\p}{\p\iota_2}\\
&+\frac{\sin\psi}{R_1-\e\eta}\bigg\{\frac{R_1\cos\varphi}{\pl_1\pl_2}\bigg(\vt_1\cdot\Big(\vt_2\times\big(\p_{\iota_1\iota_2}\vr\times\vt_2\big)\Big)\bigg)
-\sin\varphi\cos\psi\bigg\}\frac{\p}{\p\psi}\no\\
&
-\frac{\cos\psi}{R_2-\e\eta}\bigg\{\frac{R_2\cos\varphi}{\pl_1\pl_2}\bigg(\vt_2\cdot\Big(\vt_1\times\big(\p_{\iota_1\iota_2}\vr\times\vt_1\big)\Big)\bigg)
-\sin\varphi\sin\psi\bigg\}\frac{\p}{\p\psi}.\no
\end{align}

\subsubsection{Milne Problem}

Let $\bl(\eta,\iota_1,\iota_2,\vww)$ be the solution to the Milne problem
\begin{align}
\sin\varphi\frac{\p\bl}{\p\eta}+\bl-\bbl=&0,\quad\bbl(\eta,\iota_1,\iota_2)=\frac{1}{4\pi}\int_{-\pi}^{\pi}\int_{-\frac{\pi}{2}}^{\frac{\pi}{2}}\bl(\eta,\iota_1,\iota_2,\vww)\cos\varphi\ud{\varphi}\ud{\psi}, \label{expand 7}
\end{align}
with boundary condition
\begin{align}
\bl(0,\iota_1,\iota_2,\vww)=g(\iota_1,\iota_2,\vww)\ \ \text{for}\ \ \sin\varphi>0.
\end{align}
We are interested in the solution that satisfies
\begin{align}
    \lim_{\eta\rt\infty}\bl(\eta,\iota_1,\iota_2,\vww)=\bl_{\infty}(\iota_1,\iota_2)
\end{align}
for some $\bl_{\infty}(\iota_1,\iota_2)$. Based on \cite[Section 4]{AA009}, we have the well-posedness and regularity of \eqref{expand 7}.

\begin{proposition}\label{prop:boundary-wellposedness}
    Under the assumption \eqref{assumption}, there exist $\bl_{\infty}(\iota_1,\iota_2)$ and a unique solution $\bl$ to \eqref{expand 7}such that $\bll:=\bl-\bl_{\infty}$ satisfies 
    \begin{align}
        \left\{
        \begin{array}{l}
        \sin\varphi\dfrac{\p\bll}{\p\eta}+\bll-\bbll=0,\\\rule{0ex}{1.2em}
        \bll(0,\iota_1,\iota_2,\vww)=g(\iota_1,\iota_2,\vww)-\bl_{\infty}(\iota_1,\iota_2),\\\rule{0ex}{1.5em}
        \ds\lim_{\eta\rt0}\bll(\eta,\iota_1,\iota_2,\vww)=0,
        \end{array}
        \right.
    \end{align}
    and for some constant $K>0$ and any $0<\N\leq 3$
    \begin{align}
    \abs{\bl_{\infty}}_{W^{3,\infty}_{\iota_1,\iota_2}}+\lnm{\ue^{K\eta}\bll}\ls& 1,\\
    \lnm{\ue^{K\eta}\sin\varphi\frac{\p\bll}{\p\eta}}+\lnm{\ue^{K\eta}\sin\varphi\frac{\p\bll}{\p\varphi}}+\lnm{\ue^{K\eta}\frac{\p\bll}{\p\psi}}\ls& 1,\\
    \lnm{\ue^{K\eta}\frac{\p^{\N}\bll}{\p\iota_1^{\N}}}+\lnm{\ue^{K\eta}\frac{\p^{\N}\bll}{\p\iota_2^{\N}}}\ls& 1.
    \end{align}
\end{proposition}

Let $\chi(y)\in C^{\infty}(\r)$
and $\ch(y)=1-\chi(y)$ be smooth cut-off functions satisfying $\chi(y)=1$ if $\abs{y}\leq1$ and $\chi(y)=0$ if $\abs{y}\geq2$.
We define the boundary layer 
\begin{align}\label{boundary layer}
\uu_0(\eta,\iota_1,\iota_2,\vww):=\ch\big(\e^{-1}\varphi\big)\chi(\e\eta)\bll(\eta,\iota_1,\iota_2,\vww).
\end{align}

\begin{remark}
    Due to the cutoff in \eqref{boundary layer}, we have
    \begin{align}\label{expand 8}
        \uu_0(0,\iota_1,\iota_2,\vww)=\ch\big(\e^{-1}\varphi\big)\Big(g(\iota_1,\iota_2,\vww)-\bl_{\infty}(\iota_1,\iota_2)\Big)=\ch\big(\e^{-1}\varphi\big)\bll(0,\iota_1,\iota_2,\vww),
    \end{align}
    and
    \begin{align}
        \sin\varphi\dfrac{\p\uu_0}{\p\eta}+\uu_0-\overline{\uu_0}=-\overline{\ch\big(\e^{-1}\varphi\big)\chi(\e\eta)\bll}+\bbll\ch(\e^{-1}\varphi)\chi(\e\eta).
    \end{align}
\end{remark}

\subsection{Matching Procedure}

We plan to enforce the matching condition for $\vx_0\in\p\Omega$ and $\vw\cdot\vn<0$
\begin{align}
\u_0(\vx_0)+\uu_0(\vx_0,\vw)=&g(\vx_0,\vw)+O(\e).
\end{align}
Considering \eqref{expand 8}, it suffices to require
\begin{align}\label{expand 9}
    \u_0(\vx_0)=\bl_{\infty}(\vx_0):=\bl_{\infty}(\iota_1,\iota_2),
\end{align}
which yields
\begin{align}
\u_0(\vx_0)+\bll(\vx_0,\vw)=&g(\vx_0,\vw).
\end{align}
Hence, we obtain
\begin{align}
    \u_0(\vx_0,\vw)+\uu_0(\vx_0,\vw)=g(\vx_0,\vw)+\chi\big(\e^{-1}\varphi\big)\bll(0,\iota_1,\iota_2,\vww).
\end{align}

\paragraph{\underline{Construction of $\u_0$}}

Based on \eqref{expand 3} and \eqref{expand 9}, define $\u_0(\vx)$ satisfying
\begin{align}\label{expand 12}
    \u_0=\bu_0,\quad \Delta_x\bu_0=0,\quad \u_0(\vx_0)=\bl_{\infty}(\vx_0).
\end{align}
From standard elliptic estimates \cite{Krylov2008} and Proposition \ref{prop:boundary-wellposedness}, we have for any $\NN\in[2,\infty)$
\begin{align}\label{expand 10}
    \nm{\u_0}_{W^{3+\frac{1}{\NN},\NN}}+\abs{\u_0}_{W^{3,\NN}}\ls 1.
\end{align}

\paragraph{\underline{Construction of $\u_1$}}

Based on \eqref{expand 4}, define $\u_1(\vx,\vw)$ satisfying
\begin{align}\label{expand 13}
    \u_1=\bu_1-\vw\cdot\nx\u_{0},\quad\Delta_x\bu_1=0,\quad \bu_1(\vx_0)=0.
\end{align}
From \eqref{expand 10}, we have for any $\NN\in[2,\infty)$
\begin{align}\label{expand 11}
    \nm{\u_1}_{W^{2+\frac{1}{\NN},\NN}L^{\infty}}+\abs{\u_1}_{W^{2,\NN}L^{\infty}}\ls 1.
\end{align}

\paragraph{\underline{Construction of $\u_2$}}

Based on \eqref{expand 4}, define $\u_2(\vx,\vw)$ satisfying
\begin{align}\label{expand 13'}
    \u_2=\bu_2-\vw\cdot\nx\u_{1},\quad\Delta_x\bu_2=0,\quad \bu_2(\vx_0)=0.
\end{align}
From \eqref{expand 11}, we have for any $\NN\in[2,\infty)$
\begin{align}\label{expand 11'}
    \nm{\u_2}_{W^{1+\frac{1}{\NN},\NN}L^{\infty}}+\abs{\u_2}_{W^{1,\NN}L^{\infty}}\ls 1.
\end{align}

Summarizing the above analysis, we have the well-posedness and regularity estimates of the interior solution and boundary layer:

\begin{proposition}\label{prop:wellposedness}
    Under the assumption \eqref{assumption}, we can construct $\u_0,\u_1,\u_2, \uu_0$ as in \eqref{expand 12}\eqref{expand 13}\eqref{expand 13'}\eqref{boundary layer} satisfying for any $\NN\in[2,\infty)$
    \begin{align}
        \nm{\u_0}_{W^{3+\frac{1}{\NN},\NN}}+\abs{\u_0}_{W^{3,\NN}}\ls& 1,\\
        \nm{\u_1}_{W^{2+\frac{1}{\NN},\NN}L^{\infty}}+\abs{\u_1}_{W^{2,\NN}L^{\infty}}\ls& 1,\\
        \nm{\u_2}_{W^{1+\frac{1}{\NN},\NN}L^{\infty}}+\abs{\u_2}_{W^{1,\NN}L^{\infty}}\ls& 1,
    \end{align}
    and for some constant $K>0$ and any $0<\N\leq 3$
    \begin{align}
        \lnm{\ue^{K\eta}\uu_0}+\lnm{\ue^{K\eta}\frac{\p^{\N}\uu_0}{\p\iota_1^{\N}}}+\lnm{\ue^{K\eta}\frac{\p^{\N}\uu_0}{\p\iota_2^{\N}}}\ls& 1.
    \end{align}
\end{proposition}

\section{Remainder Equation}

Denote the approximate solution
\begin{align}
    \ua:=\left(\u_0+\e\u_1+\e^2\u_2\right)+\uu_0.
\end{align}
Inserting \eqref{expand} into \eqref{transport}, we have
\begin{align}
    \vw\cdot\nx\big(\ua+\re\big)+\e^{-1}\Big(\ua+\re\Big)-\e^{-1}\Big(\overline{\ua}+\bre\Big)=0,\quad \big(\ua+\re\big)\Big|_{\gamma_-}=g,
\end{align}
which yields
\begin{align}
    \vw\cdot\nx\re+\e^{-1}\Big(\re-\bre\Big)=-\vw\cdot\nx\ua-\e^{-1}\Big(\ua-\overline{\ua}\Big),\quad \re\Big|_{\gamma_-}=\big(g-\ua\big)\Big|_{\gamma_-}.
\end{align}

\subsection{Formulation of Remainder Equation}

We consider the remainder equation
\begin{align}\label{remainder}
\left\{
\begin{array}{l}\displaystyle
\vw\cdot\nabla_x \re+\e^{-1}\Big(\ire\Big)=\ss\ \ \text{in}\ \ \Omega\times\s^2,\\\rule{0ex}{1.2em}
\re(\vx_0,\vw)=\g(\vx_0,\vw)\ \ \text{for}\
\ \vw\cdot\vn<0\ \ \text{and}\ \ \vx_0\in\p\Omega,
\end{array}
\right.
\end{align}
where $\ds\bre(\vx)=\frac{1}{4\pi}\int_{\s^2}\re(\vx,\vw)\ud{\vw}$.
Here the boundary data $\g$ is given by
\begin{align}\label{d:h}
    \g:=-\e\vw\cdot\nx\u_0-\e^2\vw\cdot\nx\u_1-\chi\big(\e^{-1}\varphi\big)\bll(0),
\end{align}
and the source term $\ss$ is given by
\begin{align}
    \ss:=S_0+S_1+S_2+S_3,
\end{align}
where
\begin{align}
    S_0:=&-\e^2\vw\cdot\nx\u_2,\label{d:s0}\\
    S_1:=&\bigg(\dfrac{\sin^2\psi}{R_1-\e\eta}+\dfrac{\cos^2\psi}{R_2-\e\eta}\bigg)\cos\varphi\dfrac{\p\uu_0}{\p\varphi},\label{d:s1}\\
    S_2:=&\e^{-1}\sin\phi\ch\big(\e^{-1}\varphi\big)\frac{\p\chi(\e\eta)}{\p\eta}\bll+\frac{R_1\cos\varphi\sin\psi}{\pl_1(R_1-\e\eta)}\frac{\p\uu_0}{\p\iota_1}+\frac{R_2\cos\varphi\cos\psi}{\pl_2(R_2-\e\eta)}\frac{\p\uu_0}{\p\iota_2}\label{d:s2}\\
    &+\frac{\sin\psi}{R_1-\e\eta}\bigg\{\frac{R_1\cos\varphi}{\pl_1\pl_2}\bigg(\vt_1\cdot\Big(\vt_2\times\big(\p_{\iota_1\iota_2}\vr\times\vt_2\big)\Big)\bigg)
    -\sin\varphi\cos\psi\bigg\}\frac{\p\uu_0}{\p\psi}\no\\
    &
    -\frac{\cos\psi}{R_2-\e\eta}\bigg\{\frac{R_2\cos\varphi}{\pl_1\pl_2}\bigg(\vt_2\cdot\Big(\vt_1\times\big(\p_{\iota_1\iota_2}\vr\times\vt_1\big)\Big)\bigg)
    -\sin\varphi\sin\psi\bigg\}\frac{\p\uu_0}{\p\psi},\no\\
    S_3:=&\e^{-1}\bigg(\overline{\ch\big(\e^{-1}\varphi\big)\chi(\e\eta)\bll}-\bbll\ch\big(\e^{-1}\varphi\big)\chi(\e\eta)\bigg).\label{d:s3}
\end{align}

\subsection{Weak Formulation}

\begin{lemma}[Green's Identity, Lemma 2.2 of \cite{Esposito.Guo.Kim.Marra2013}]\label{lem:green}
Assume $f(\vx,\vw),\ g(\vx,\vw)\in L^2(\Omega\times\s^2)$ and
$\vw\cdot\nx f,\ \vw\cdot\nx g\in L^2(\Omega\times\s^2)$ with $f,\
g\in L^2_{\gamma}$. Then
\begin{align}
\iint_{\Omega\times\s^2}\Big(\big(\vw\cdot\nx f\big)g+\big(\vw\cdot\nx
g\big)f\Big)\ud{\vx}\ud{\vw}=\int_{\gamma}fg(\vw\cdot n)=\int_{\gamma_+}fg\ud{\gamma}-\int_{\gamma_-}fg\ud{\gamma}.
\end{align}
\end{lemma}

Using Lemma \ref{lem:green}, we can derive the weak formulation of \eqref{remainder}. For any test function $\weak(x,w)\in L^2(\Omega\times\s^2)$ with $\vw\cdot\nx\weak\in L^2(\Omega\times\s^2)$ with $\weak\in L^2_{\gamma}$, we have
\begin{align}\label{weak formulation}
    \int_{\gamma}\re\weak(\vw\cdot n)-\iint_{\Omega\times\s^2}\re\big(\vw\cdot\nx \weak\big)+\e^{-1}\iint_{\Omega\times\s^2}\Big(\re-\bre\Big)\weak=\iint_{\Omega\times\s^2}\ss\weak.
\end{align}

\subsection{Estimates of Boundary and Source Terms}

\begin{lemma}\label{h-estimate}
    Under the assumption \eqref{assumption}, for $\g$ defined in \eqref{d:h}, we have
    \begin{align}
    \tnms{\g}{\gamma_-}\ls\e.
    \end{align}
\end{lemma}
\begin{proof}
Based on Proposition \ref{prop:wellposedness}, we have
\begin{align}
    \tnms{\e\vw\cdot\nx\u_0}{\gamma_-}+\tnms{\e^2\vw\cdot\nx\u_1}{\gamma_-}\ls\e.
\end{align}
Noting the cutoff $\chi\big(\e^{-1}\varphi\big)$ restricts the support to $\abs{\varphi}\ls\e$ and $\ud\gamma$ measure contributes an extra $\sin\varphi$, we have
\begin{align}
    \tnms{\chi\big(\e^{-1}\varphi\big)\bll(0)}{\gamma_-}\ls\e.
\end{align}
Hence, our result follows.
\end{proof}

\begin{lemma}\label{s0-estimate}
    Under the assumption \eqref{assumption}, for $\ss_0$ defined in \eqref{d:s0}, we have
    \begin{align}
        \tnm{\ss_0}\ls\e^2.
    \end{align}
\end{lemma}
\begin{proof}
    This follows from Proposition \ref{prop:wellposedness}.
\end{proof}

\begin{lemma}\label{s1-estimate}
    Under the assumption \eqref{assumption}, for $\ss_1$ defined in \eqref{d:s1}, we have
    \begin{align}\label{final 17}
        &\tnm{\big(1+\eta\big)\ss_{1}}\ls1.
    \end{align}
    Also, for the boundary layer $\uu_0$ defined in \eqref{boundary layer}, we have
    \begin{align}\label{final 16}
        \tnm{\big(1+\eta\big)\uu_0}\ls \e^{\frac{1}{2}},\quad \nm{\big(1+\eta\big)\uu_0}_{L^2_xL^1_w}\ls \e^{\frac{1}{2}},
    \end{align}
    and
    \begin{align}\label{final 11}
        \abs{\bbr{\big(1+\eta\big)\ss_1,g}}&\ls\tnm{\big(1+\eta\big)\br{v}^2\uu_0}\tnm{\nabla_w g}\ls\e^{\frac{1}{2}}\tnm{\nabla_w g}.
    \end{align}
\end{lemma}
\begin{proof}
    We split
    \begin{align}
    S_1=S_{11}+S_{12}
    :=&\bigg(\dfrac{\sin^2\psi}{R_1-\e\eta}+\dfrac{\cos^2\psi}{R_2-\e\eta}\bigg)\cos\varphi\dfrac{\p\bll}{\p\varphi}\ch\big(\e^{-1}\varphi\big)\chi(\e\eta)\\
    &+\bigg(\dfrac{\sin^2\psi}{R_1-\e\eta}+\dfrac{\cos^2\psi}{R_2-\e\eta}\bigg)\cos\varphi\dfrac{\p\ch\big(\e^{-1}\varphi\big)}{\p\varphi}\chi(\e\eta)\bll.\no
    \end{align}
    Note that $\ss_{11}$ is nonzero only when $\abs{\varphi}\geq\e$ and thus based on Proposition \ref{prop:boundary-wellposedness}, we know $\abs{\dfrac{\p\bll}{\p\varphi}}\leq\abs{\sin\varphi}^{-1}\abs{\bll}\ls\e^{-1}$. Hence, using $\ud\mu=\e\ud\eta$, we have
    \begin{align}\label{final 09}
        \tnm{\ss_{11}}\ls&\left(\iint_{\abs{\varphi}\geq\e}\abs{\dfrac{\p\bll}{\p\varphi}}^{2}\ud\varphi\ud\mu\right)^{\frac{1}{2}}\ls \left(\iint_{\abs{\varphi}\geq\e}\abs{\sin\varphi}^{-2}\abs{\bll}^2\ud\varphi\ud\mu\right)^{\frac{1}{2}}\\
        \ls&\left(\iint_{\abs{\varphi}\geq\e}\abs{\sin\varphi}^{-2}\ue^{-2K\eta}\ud\varphi\ud\mu\right)^{\frac{1}{2}}\ls \left(\e\iint_{\abs{\varphi}\geq\e}\abs{\sin\varphi}^{-2}\ue^{-2K\eta}\ud\varphi\ud\eta\right)^{\frac{1}{2}}
        \ls\left(\e\e^{-1}\right)^{\frac{1}{2}}=1.\no
    \end{align}
    Noticing $\dfrac{\p\ch\big(\e^{-1}\varphi\big)}{\p\varphi}=\e^{-1}\ch'\big(\e^{-1}\varphi\big)$, and $\ch'\big(\e^{-1}\varphi\big)$ is nonzero only when $\e<\abs{\varphi}<2\e$, based on Proposition \ref{prop:boundary-wellposedness}, we have
    \begin{align}\label{final 10}
        \tnm{\ss_{12}}\ls&\e^{-1}\left(\iint_{\e<\abs{\varphi}<2\e}\abs{\bll}^{2}\ud\varphi\ud\mu\right)^{\frac{1}{2}}\ls\e^{-1}\left(\iint_{\e<\abs{\varphi}<2\e}\ue^{-2K\eta}\ud\varphi\ud\mu\right)^{\frac{1}{2}}\\
        \ls&\e^{-1}\left(\e\iint_{\e<\abs{\varphi}<2\e}\ue^{-2K\eta}\ud\varphi\ud\eta\right)^{\frac{1}{2}}\ls\e^{-1}\left(\e\e\right)^{\frac{1}{2}}=1.\no
    \end{align}
    Collecting \eqref{final 09} and \eqref{final 10}, we have \eqref{final 17}. Note that $\ue^{-K\eta}$ will suppress the growth from the pre-factor $1+\eta$.
    
    \eqref{final 16} comes from Proposition \ref{prop:boundary-wellposedness}. Then we turn to \eqref{final 11}. The most difficult term in $\babs{\br{\ss_1,g}}$ is essentially $\abs{\br{\dfrac{\p\uu_0}{\p\varphi},g}}$. Integration by parts with respect to $\varphi$ implies
    \begin{align}
   \abs{\br{\frac{\p\uu_0}{\p\varphi},g}}\ls& \abs{\br{\uu_0, \frac{\p g}{\p\varphi}}}
   \ls\tnm{\uu_0}\tnm{\frac{\p g}{\p\varphi}}.
    \end{align}
    From \eqref{velocity} and $\dfrac{\p\vx}{\p\varphi}=0$, we know the substitution $(\mn,\iota_1,\iota_2,\vw)\rt (\mn,\iota_1,\iota_2,\vww)$ implies
    \begin{align}
    -\dfrac{\p\vw}{\p\varphi}\cdot\vn=\cos\varphi,\quad 
    \dfrac{\p\vw}{\p\varphi}\cdot\vt_1=-\sin\varphi\sin\psi,\quad
    \dfrac{\p\vw}{\p\varphi}\cdot\vt_2=-\sin\varphi\cos\psi.
    \end{align}
    Hence, we know $\abs{\dfrac{\p\vw}{\p\varphi}}\ls 1$,
    and thus
    \begin{align}
    \abs{\frac{\p g}{\p\varphi}}\ls \abs{\nabla_w g}\abs{\frac{\p \vw}{\p\varphi}}\ls\abs{\nabla_w g}.
    \end{align}
    Hence, we know that 
    \begin{align}
    \abs{\br{\frac{\p\uu_0}{\p\varphi},g}}\ls&\tnm{\uu_0}\tnm{\nabla_w g}\ls\e^{\frac{1}{2}}\tnm{\nabla_w g}.
    \end{align}
\end{proof}

\begin{lemma}\label{s2-estimate}
    Under the assumption \eqref{assumption}, for $\ss_2$ defined in \eqref{d:s2}, we have
    \begin{align}
        &\tnm{\big(1+\eta\big)\ss_{2}}\ls \e^{\frac{1}{2}},\quad \nm{\big(1+\eta\big)\ss_2}_{L^2_xL^1_w}\ls \e^{\frac{1}{2}}.
    \end{align}
\end{lemma}
\begin{proof}
    Notice that $\abs{\e^{-1}\sin\phi\ch\big(\e^{-1}\varphi\big)\dfrac{\p\chi(\e\eta)}{\p\eta}}\ls 1$. Based on Proposition \ref{prop:boundary-wellposedness} and Proposition \ref{prop:wellposedness}, we directly bound
    \begin{align}
        \tnm{\ss_2}\ls&\left(\iint\bigg(\abs{\bl}^2+\abs{\dfrac{\p\bl}{\p\iota_1}}^{2}+\abs{\dfrac{\p\bl}{\p\iota_2}}^{2}+\abs{\dfrac{\p\bl}{\p\psi}}^{2}\bigg)\ud\varphi\ud\mu\right)^{\frac{1}{2}}\\
        \ls&\left(\iint\ue^{-2K\eta}\ud\varphi\ud\mu\right)^{\frac{1}{2}}\ls \left(\e\iint\ue^{-2K\eta}\ud\varphi\ud\eta\right)^{\frac{1}{2}}\ls\e^{\frac{1}{2}}.\no
    \end{align}
    Then the $L^2_xL^1_w$ estimate follows from a similar argument noting that there is no rescaling in $\vw$ variables.
\end{proof}

\begin{lemma}\label{s3-estimate}
    Under the assumption \eqref{assumption}, for $\ss_3$ defined in \eqref{d:s3}, we have
    \begin{align}
        &\tnm{\big(1+\eta\big)\ss_{3}}\ls1,\quad \nm{\big(1+\eta\big)\ss_3}_{L^2_xL^1_w}\ls \e^{\frac{1}{2}}.
    \end{align}
\end{lemma}
\begin{proof}
    Using $\chi=1-\ch$, we split
    \begin{align}
    S_3=S_{31}+S_{32}:=&\e^{-1}\bbll\chi\big(\e^{-1}\varphi\big)\chi(\e\eta)-\e^{-1}\overline{\chi\big(\e^{-1}\varphi\big)\chi(\e\eta)\bll}.
    \end{align}
    Noting that $\ss_{31}$ is nonzero only when $\abs{\varphi}\leq\e$, based on Proposition \ref{prop:boundary-wellposedness}, we have
    \begin{align}\label{final 07}
        \tnm{\ss_{31}}\ls&\left(\iint_{\abs{\varphi}\leq\e}\abs{\e^{-1}\bbll}^{2}\ud\varphi\ud\mu\right)^{\frac{1}{2}}\ls \left(\e^{-2}\iint_{\abs{\varphi}\leq\e}\ue^{-2K\eta}\ud\varphi\ud\mu\right)^{\frac{1}{2}}\\
        \ls&\left(\e^{-1}\iint_{\abs{\varphi}\leq\e}\ue^{-2K\eta}\ud\varphi\ud\eta\right)^{\frac{1}{2}}\ls \left(\e^{-1}\e\right)^{\frac{1}{2}}\ls1.\no
    \end{align}
    Analogously, noting that $\ss_{32}$ contains $\vw$ integral, we have
    \begin{align}\label{final 08}
        \tnm{\ss_{32}}\ls&\left(\iint\abs{\e^{-1}\overline{\bll\chi(\e^{-1}\varphi)}}^{2}\ud\varphi\ud\mu\right)^{\frac{1}{2}}\ls \left(\e^{-2}\iint\abs{\int_{\abs{\varphi}\leq\e}\bll\ud\varphi}^2\ud\varphi\ud\mu\right)^{\frac{1}{2}}\\
        \ls&\left(\e^{-2}\iint\abs{\int_{\abs{\varphi}\leq\e}\ue^{-K\eta}\ud\varphi}^2\ud\varphi\ud\mu\right)^{\frac{1}{2}}\ls\left(\e^{-2}\iint\e^2\ue^{-2K\eta}\ud\varphi\ud\mu\right)^{\frac{1}{2}}\no\\
        \ls&\left(\iint\ue^{-2K\eta}\ud\varphi\ud\mu\right)^{\frac{1}{2}}\ls\left(\e\iint\ue^{-2K\eta}\ud\varphi\ud\eta\right)^{\frac{1}{2}}\ls\e^{\frac{1}{2}}.\no
    \end{align}
    Collecting \eqref{final 07} and \eqref{final 08}, we have the $L^2$ estimate. Similarly, we derive the $L^2_xL^1_w$ bound:
    \begin{align}
        \nm{\ss_{31}}_{L^2_xL^1_w}\ls&\left(\int\bigg(\int_{\abs{\varphi}\leq\e}\abs{\e^{-1}\bbll}\ud\varphi\bigg)^2\ud\mu\right)^{\frac{1}{2}}\ls \left(\int\ue^{-2K\eta}\ud\mu\right)^{\frac{1}{2}}
        \ls\left(\e\int\ue^{-2K\eta}\ud\eta\right)^{\frac{1}{2}}\ls\e^{\frac{1}{2}},\\
        \nm{\ss_{32}}_{L^2_xL^1_w}\ls&\left(\int\bigg(\int\abs{\e^{-1}\overline{\bll\chi(\e^{-1}\varphi)}}\ud\varphi\bigg)^2\ud\mu\right)^{\frac{1}{2}}\ls \left(\e^{-2}\int\bigg(\int\abs{\int_{\abs{\varphi}\leq\e}\bll\ud\varphi}\ud\varphi\bigg)^2\ud\mu\right)^{\frac{1}{2}}\\
        \ls&\left(\e^{-2}\int\bigg(\int\e\ue^{-K\eta}\ud\varphi\bigg)^2\ud\mu\right)^{\frac{1}{2}}\ls \left(\int\ue^{-2K\eta}\ud\mu\right)^{\frac{1}{2}}
        \ls\left(\e\int\ue^{-2K\eta}\ud\eta\right)^{\frac{1}{2}}\ls\e^{\frac{1}{2}}.\no
    \end{align}
\end{proof}

\section{Remainder Estimate}

\subsection{Basic Energy Estimate}

\begin{lemma}
    Under the assumption \eqref{assumption}, we have 
    \begin{align}\label{energy}
    \e^{-1}\tnms{\re}{\gamma_+}^2+\e^{-2}\tnm{\ire}^2\ls \oo\e^{-1}\tnm{\bre}^2+1.
    \end{align}
\end{lemma}

\begin{proof}
Taking $\weak=\e^{-1}\re$ in \eqref{weak formulation}, we obtain
\begin{align}
    \frac{\e^{-1}}{2}\int_{\gamma}\abs{\re}^2(\vw\cdot\vn)+\e^{-2}\bbr{\re,\ire}=\e^{-1}\bbr{\re,\ss}.
\end{align}
Then using the orthogonality of $\bre$ and $\ire$, we have
\begin{align}
    \frac{\e^{-1}}{2}\tnms{\re}{\gamma_+}^2+\e^{-2}\tnm{\ire}^2=\e^{-1}\bbr{\re,\ss}+\frac{\e^{-1}}{2}\tnms{\g}{\gamma_-}^2.
\end{align}
Using Lemma \ref{h-estimate}, we know
\begin{align}\label{final 06}
    \e^{-1}\tnms{\re}{\gamma_+}^2+\e^{-2}\tnm{\ire}^2\ls\e+\e^{-1}\bbr{\re, \ss_0+\ss_1+\ss_2+\ss_3}.
\end{align}
Using Lemma \ref{s0-estimate}, we have
\begin{align}\label{final 01}
    \abs{\e^{-1}\bbr{\re, \ss_0}}\ls \e^{-1}\tnm{\re}\tnm{\ss_0}\ls\e\tnm{\re}\ls \oo\tnm{\re}^2+\e^2.
\end{align}
Using Lemma \ref{s1-estimate}, Lemma \ref{s2-estimate} and Lemma \ref{s3-estimate}, we have
\begin{align}\label{final 02}
    \abs{\e^{-1}\bbr{\ire, \ss_{1}+\ss_{2}+\ss_{3}}}
    \ls&\e^{-1}\tnm{\ire}\tnm{\ss_{1}+\ss_{2}+\ss_{3}}\\
    \ls&\e^{-1}\tnm{\ire}
    \ls \oo\e^{-2}\tnm{\ire}^2+1.\no
\end{align}
Finally, we turn to $\e^{-1}\bbr{\bre, \ss_1+\ss_2+\ss_3}$. For $\ss_1$, we integrate by parts with respect to $\varphi$ and use Lemma \ref{s1-estimate} to obtain
\begin{align}\label{final 03}
    \abs{\e^{-1}\bbr{\bre, \ss_1}}=&\e^{-1}\abs{\br{\bre, \bigg(\dfrac{\sin^2\psi}{R_1-\e\eta}+\dfrac{\cos^2\psi}{R_2-\e\eta}\bigg)\cos\varphi\dfrac{\p\uu_0}{\p\varphi}}}\\
    =&\e^{-1}\abs{\br{\bre, \bigg(\dfrac{\sin^2\psi}{R_1-\e\eta}+\dfrac{\cos^2\psi}{R_2-\e\eta}\bigg)\uu_0\sin\varphi}}\no\\
    \ls&\e^{-1}\nm{\bre}_{L^2}\nm{\uu_1}_{L^2_xL^1_w}\ls\e^{-\frac{1}{2}}\nm{\bre}_{L^2}\ls \oo\e^{-1}\tnm{\bre}^2+1.\no
\end{align}
Also, Lemma \ref{s2-estimate} and Lemma \ref{s3-estimate} yield
\begin{align}\label{final 04}
    \abs{\e^{-1}\bbr{\bre, \ss_2+\ss_3}}
    \ls&\e^{-1}\nm{\bre}_{L^2}\Big(\nm{\ss_2}_{L^2_xL^1_w}+\nm{\ss_3}_{L^2_xL^1_w}\Big)\ls\e^{-\frac{1}{2}}\nm{\bre}_{L^2}\ls \oo\e^{-1}\tnm{\bre}^2+1.
\end{align}
Collecting \eqref{final 01}\eqref{final 02}\eqref{final 03}\eqref{final 04}, we obtain
\begin{align}\label{final 05}
    \abs{\e^{-1}\bbr{\re, \ss_0+\ss_1+\ss_2+\ss_3}}\ls \oo\e^{-2}\tnm{\ire}^2+\oo\e^{-1}\tnm{\re}^2+1.
\end{align}
Combining \eqref{final 05} and \eqref{final 06}, we have \eqref{energy}.
\end{proof}

\subsection{Kernel Estimate}

\begin{lemma}
    Under the assumption \eqref{assumption}, we have
    \begin{align}\label{kernel}
    \tnm{\bre}^2\ls \tnm{\ire}^2+\tnms{\re}{\gamma_+}^2+\e.
    \end{align}
\end{lemma}

\begin{proof}
Denote $\test(x)$ satisfying
\begin{align}
\left\{
\begin{array}{l}
-\Delta_x\test=\bre\ \ \text{in}\
\ \Omega,\\\rule{0ex}{1.2em}
\test(\vx_0)=0\ \ \text{on}\ \
\p\Omega.
\end{array}
\right.
\end{align}
Based on standard elliptic estimates and trace estimates, we have
\begin{align}\label{kk 15}
    \nm{\test}_{H^2}+\abs{\test}_{H^{\frac{3}{2}}}\ls\tnm{\bre}.
\end{align}
Taking $\weak=\test$ in \eqref{weak formulation}, we have
\begin{align}
    \int_{\gamma}\re\test(\vw\cdot n)-\bbr{\re,w\cdot\nx\test}+\e^{-1}\bbr{\ire, \test}=\bbr{\ss, \test}.
\end{align}
Using oddness, orthogonality and $\test\big|_{\p\Omega}=0$, we obtain \eqref{kk 01}.

Then taking $\weak=w\cdot\nx\test$ in \eqref{weak formulation}, we obtain \eqref{kk 02}.

Adding  $\e^{-1}\times$\eqref{kk 01} and \eqref{kk 02} to eliminate $\e^{-1}\bbr{\ire, w\cdot\nx\test}$, we obtain
\begin{align}\label{kk 03}
    \int_{\gamma}\re\big(w\cdot\nx\test\big)(\vw\cdot n)-\bbr{\re,w\cdot\nx\big(w\cdot\nx\test\big)}
    =&\e^{-1}\bbr{\ss,\test}+\bbr{\ss, w\cdot\nx\test}.
\end{align}
Notice that
\begin{align}\label{kk 04}
    -\bbr{\re,w\cdot\nx\big(w\cdot\nx\test\big)}=&-\bbr{\bre,w\cdot\nx\big(w\cdot\nx\test\big)}-\bbr{\ire,w\cdot\nx\big(w\cdot\nx\test\big)},
\end{align}
where \eqref{kk 15} and Cauchy's inequality yield
\begin{align}
    -\bbr{\bre,w\cdot\nx\big(w\cdot\nx\test\big)}\simeq&\tnm{\bre}^2,\\
    \abs{\bbr{\ire,w\cdot\nx\big(w\cdot\nx\test\big)}}\ls&\tnm{\ire}^2+\oo\tnm{\bre}^2.
\end{align}
Also, using \eqref{kk 15} and Lemma \ref{h-estimate}, we have
\begin{align}\label{kk 06}
    \abs{\int_{\gamma}\re\big(w\cdot\nx\test\big)(\vw\cdot n)}\ls&\Big(\tnms{\re}{\gamma_+}+\tnms{\g}{\gamma_-}\Big)\abs{\nx\test}_{L^2}
    \ls\oo\tnm{\bre}^2+\tnms{\re}{\gamma_+}^2+\e^2.
\end{align}
Inserting \eqref{kk 04}--\eqref{kk 06} into \eqref{kk 03}, we obtain
\begin{align}\label{kk 07}
    \tnm{\bre}^2\ls&\e^2+ \tnm{\ire}^2+\tnms{\re}{\gamma_+}^2+\abs{\e^{-1}\bbr{\ss,\test}}+\abs{\bbr{\ss, w\cdot\nx\test}}.
\end{align}
Then we turn to the estimate of source terms in \eqref{kk 07}.
Cauchy's inequality and Lemma \ref{s0-estimate} yield
\begin{align}\label{kk 08}
    \abs{\e^{-1}\bbr{\ss_0, \test}}+\abs{\bbr{\ss_0, w\cdot\nx\test}}\ls\e^{-1}\tnm{\ss_0}\nm{\test}_{H^1}\ls\e\tnm{\bre}\ls \oo\tnm{\bre}^2+\e^2.
\end{align}
Similar to \eqref{final 03}, we first integrate by parts with respect to $\varphi$ in $\ss_1$. Using $\test\big|_{\p\Omega}=0$, \eqref{kk 15}, Hardy's inequality and Lemma \ref{s1-estimate}, Lemma \ref{s2-estimate}, Lemma \ref{s3-estimate}, we have
\begin{align}\label{kk 09}
    &\abs{\e^{-1}\bbr{\ss_1+\ss_2+\ss_3, \test}}
    \ls\abs{\e^{-1}\bbr{\uu_0+\ss_2+\ss_3,\int_0^{\mu}\frac{\p\test}{\p\mu}}}
    =\abs{\br{\eta\uu_0+\eta\ss_2+\eta\ss_3, \frac{1}{\mu}\int_0^{\mu}\frac{\p\test}{\p\mu}}}\\
    \ls&\nm{\eta\uu_0+\eta\ss_2+\eta\ss_3}_{L^2_xL^1_w}\tnm{\frac{1}{\mu}\int_0^{\mu}\frac{\p\test}{\p\mu}}
    \ls\nm{\eta\uu_0+\eta\ss_2+\eta\ss_3}_{L^2_xL^1_w}\tnm{\frac{\p\test}{\p\mu}} \ls\e^{\frac{1}{2}}\nm{\test}_{H^1}\no\\
    \ls& \e^{\frac{1}{2}}\tnm{\bre}
    \ls\oo\tnm{\bre}^2+\e.\no
\end{align}
Analogously, we integrate by parts with respect to $\varphi$ in $\ss_1$. Then using \eqref{kk 15}, fundamental theorem of calculus, Hardy's inequality and Lemma \ref{s1-estimate}, Lemma \ref{s2-estimate}, Lemma \ref{s3-estimate}, we bound
\begin{align}\label{kk 10}
    &\abs{\bbr{\ss_1+\ss_2+\ss_3, w\cdot\nx\test}}\ls\abs{\br{\uu_0+\ss_2+\ss_3, \nx\test\Big|_{\mu=0}+\int_0^{\mu}\frac{\p\big(\nx\test\big)}{\p\mu}}} \\
    \ls&\abs{\br{\uu_0+\ss_2+\ss_3, \nx\test\Big|_{\mu=0}}}+\abs{\e\br{\eta\uu_0+\eta\ss_2+\eta\ss_3, \frac{1}{\mu}\int_0^{\mu}\frac{\p\big(\nx\test\big)}{\p\mu}}}\no\\
    \ls&\nm{\uu_0+\ss_2+\ss_3}_{L^2_xL^1_w}\abs{\nx\test}_{L^2}+\e\tnm{\eta\uu_0+\eta\ss_2+\eta\ss_3}\tnm{\frac{\p\big(\nx\test\big)}{\p\mu}}\no\\
    \ls&\e^{\frac{1}{2}}\tnms{\nx\test}{\p\Omega}+\e\nm{\test}_{H^2}\ls\e^{\frac{1}{2}}\tnm{\bre}
    \ls\oo\tnm{\bre}^2+\e.\no
\end{align}
Hence, inserting \eqref{kk 08}, \eqref{kk 09} and \eqref{kk 10} into \eqref{kk 07}, we have shown \eqref{kernel}.
\end{proof}

\subsection{Synthesis}

\begin{proposition}\label{prop:energy}
    Under the assumption \eqref{assumption}, we have
    \begin{align}
    \e^{-\frac{1}{2}}\tnms{\re}{\gamma_+}+\e^{-\frac{1}{2}}\tnm{\bre}+\e^{-1}\tnm{\ire}\ls1.
    \end{align}
\end{proposition}

\begin{proof}
From \eqref{energy}, we have
\begin{align}\label{kk 11}
    \e^{-1}\tnms{\re}{\gamma_+}^2+\e^{-2}\tnm{\ire}^2\ls \oo\e^{-1}\tnm{\bre}^2+1.
\end{align}
From \eqref{kernel}, we have
\begin{align}\label{kk 12}
    \tnm{\bre}^2\ls \tnm{\ire}^2+\tnms{\re}{\gamma_+}^2+\e.
\end{align}
Inserting \eqref{kk 12} into \eqref{kk 11}, we have
\begin{align}\label{kk 16}
    \e^{-1}\tnms{\re}{\gamma_+}^2+\e^{-2}\tnm{\ire}^2\ls 1.
\end{align}
Inserting \eqref{kk 16} into \eqref{kk 12}, we have
\begin{align}\label{kk 17}
    \tnm{\bre}^2\ls\e.
\end{align}
Hence, adding $\e^{-1}\times$\eqref{kk 17} and \eqref{kk 16}, we have 
\begin{align}
    \e^{-1}\tnms{\re}{\gamma_+}^2+\e^{-1}\tnm{\bre}^2+\e^{-2}\tnm{\ire}^2\ls1.
\end{align}
Then our result follows.
\end{proof}

\section{Proof of Main Theorem}

The well-posedness of \eqref{transport} is well-known \cite{Bensoussan.Lions.Papanicolaou1979, Bardos.Santos.Sentis1984, AA003}. The construction of $\u_0$, $\bl$ and $\bl_{\infty}$ follows from Proposition \ref{prop:boundary-wellposedness} and Proposition \ref{prop:wellposedness}, so we focus on the derivation of \eqref{main}.
    
Based on Proposition \ref{prop:energy} and \eqref{expand}, we have
\begin{align}\label{rr 04}
    \tnm{u^{\e}-\u_0-\e\u_1-\e^2\u_2-\uu_0}\ls\e^{\frac{1}{2}}.
\end{align}
Using Proposition \ref{prop:wellposedness}, we have
\begin{align}\label{rr 05}
    \tnm{\e\u_1+\e^2\u_2}\ls \e.
\end{align}
Using Proposition \ref{prop:wellposedness} and the rescaling $\eta=\e^{-1}\mn$, we have
\begin{align}\label{rr 06}
    \tnm{\uu_0}\ls\e^{\frac{1}{2}}.
\end{align}
Then \eqref{main} follows from inserting \eqref{rr 05}\eqref{rr 06} into \eqref{rr 04}.

\bibliographystyle{siam}
\bibliography{Reference}

\begin{thebibliography}{10}

\bibitem{Bardos.Golse.Perthame1987}
{\sc C.~Bardos, F.~Golse, and B.~Perthame}, {\em The {Rosseland} approximation
  for the radiative transfer equations}, Comm. Pure Appl. Math., 40 (1987),
  pp.~69--721.

\bibitem{Bardos.Golse.Perthame.Sentis1988}
{\sc C.~Bardos, F.~Golse, B.~Perthame, and R.~Sentis}, {\em The nonaccretive
  radiative transfer equations: existence of solutions and {Rosseland}
  approximation}, J. Funct. Anal., 77 (1988), pp.~434--460.

\bibitem{Bardos.Phung2017}
{\sc C.~Bardos and K.~D. Phung}, {\em Observation estimate for kinetic
  transport equations by diffusion approximation}, C. R. Math. Acad. Sci.
  Paris, 355 (2017), pp.~640--664.

\bibitem{Bardos.Santos.Sentis1984}
{\sc C.~Bardos, R.~Santos, and R.~Sentis}, {\em Diffusion approximation and
  computation of the critical size}, Trans. Amer. Math. Soc., 284 (1984),
  pp.~617--649.

\bibitem{Bensoussan.Lions.Papanicolaou1979}
{\sc A.~Bensoussan, J.-L. Lions, and G.~C. Papanicolaou}, {\em Boundary layers
  and homogenization of transport processes}, Publ. Res. Inst. Math. Sci., 15
  (1979), pp.~53--157.

\bibitem{Esposito.Guo.Kim.Marra2013}
{\sc R.~Esposito, Y.~Guo, C.~Kim, and R.~Marra}, {\em Non-isothermal boundary
  in the {Boltzmann} theory and {Fourier} law}, Comm. Math. Phys., 323 (2013),
  pp.~177--239.

\bibitem{AA007}
{\sc Y.~Guo and L.~Wu}, {\em Geometric correction in diffusive limit of neutron
  transport equation in {2D} convex domains}, Arch. Rational Mech. Anal., 226
  (2017), pp.~321--403.

\bibitem{AA009}
\leavevmode\vrule height 2pt depth -1.6pt width 23pt, {\em Regularity of
  {Milne} problem with geometric correction in {3D}}, Math. Models Methods
  Appl. Sci., 27 (2017), pp.~453--524.

\bibitem{Krylov2008}
{\sc N.~V. Krylov}, {\em Lectures on elliptic and parabolic equations in
  {Sobolev} spaces.}, American Mathematical Society, Providence, RI, 2008.

\bibitem{Larsen1974=}
{\sc E.~W. Larsen}, {\em A functional-analytic approach to the steady,
  one-speed neutron transport equation with anisotropic scattering}, Comm. Pure
  Appl. Math., 27 (1974), pp.~523--545.

\bibitem{Larsen1974}
\leavevmode\vrule height 2pt depth -1.6pt width 23pt, {\em Solutions of the
  steady, one-speed neutron transport equation for small mean free paths}, J.
  Mathematical Phys., 15 (1974), pp.~299--305.

\bibitem{Larsen1975}
\leavevmode\vrule height 2pt depth -1.6pt width 23pt, {\em Neutron transport
  and diffusion in inhomogeneous media {I}}, J. Mathematical Phys., 16 (1975),
  pp.~1421--1427.

\bibitem{Larsen1977}
\leavevmode\vrule height 2pt depth -1.6pt width 23pt, {\em Asymptotic theory of
  the linear transport equation for small mean free paths {II}}, SIAM J. Appl.
  Math., 33 (1977), pp.~427--445.

\bibitem{Larsen.D'Arruda1976}
{\sc E.~W. Larsen and J.~D'Arruda}, {\em Asymptotic theory of the linear
  transport equation for small mean free paths {I}}, Phys. Rev., 13 (1976),
  pp.~1933--1939.

\bibitem{Larsen.Habetler1973}
{\sc E.~W. Larsen and G.~J. Habetler}, {\em A functional-analytic derivation of
  {Case}'s full and half-range formulas}, Comm. Pure Appl. Math., 26 (1973),
  pp.~525--537.

\bibitem{Larsen.Keller1974}
{\sc E.~W. Larsen and J.~B. Keller}, {\em Asymptotic solution of neutron
  transport problems for small mean free paths}, J. Mathematical Phys., 15
  (1974), pp.~75--81.

\bibitem{Larsen.Zweifel1974}
{\sc E.~W. Larsen and P.~F. Zweifel}, {\em On the spectrum of the linear
  transport operator}, J. Mathematical Phys., 15 (1974), pp.~1987--1997.

\bibitem{Larsen.Zweifel1976}
\leavevmode\vrule height 2pt depth -1.6pt width 23pt, {\em Steady,
  one-dimensional multigroup neutron transport with anisotropic scattering}, J.
  Mathematical Phys., 17 (1976), pp.~1812--1820.

\bibitem{Li.Lu.Sun2015}
{\sc Q.~Li, J.~Lu, and W.~Sun}, {\em Diffusion approximations and domain
  decomposition method of linear transport equations: asymptotics and
  numerics}, J. Comput. Phys., 292 (2015), pp.~141--167.

\bibitem{Li.Lu.Sun2015==}
\leavevmode\vrule height 2pt depth -1.6pt width 23pt, {\em Half-space kinetic
  equations with general boundary conditions}, Math. Comp., 86 (2017),
  pp.~1269--1301.

\bibitem{Li.Lu.Sun2017}
\leavevmode\vrule height 2pt depth -1.6pt width 23pt, {\em Validity and
  regularization of classical half-space equations}, J. Stat. Phys., 166
  (2017), pp.~398--433.

\bibitem{AA014}
{\sc L.~Wu}, {\em Boundary layer of transport equation with in-flow boundary},
  Arch. Rational Mech. Anal., 235 (2020), pp.~2085--2169.

\bibitem{AA016}
\leavevmode\vrule height 2pt depth -1.6pt width 23pt, {\em Diffusive limit of
  transport equation in {3D} convex domains}, Peking Math. J., 4 (2021),
  pp.~203--284.

\bibitem{AA003}
{\sc L.~Wu and Y.~Guo}, {\em Geometric correction for diffusive expansion of
  steady neutron transport equation}, Comm. Math. Phys., 336 (2015),
  pp.~1473--1553.

\bibitem{AA006}
{\sc L.~Wu, X.~Yang, and Y.~Guo}, {\em Asymptotic analysis of transport
  equation in annulus}, J. Stat. Phys., 165 (2016), pp.~585--644.

\end{thebibliography}

\end{document}